\documentclass[11pt]{amsart}

\usepackage{enumerate}
\usepackage{amsmath, amsthm, amsfonts, amssymb, xy,  mathrsfs, graphicx, lscape}
\usepackage[usenames, dvipsnames]{color}
\usepackage[margin=1in]{geometry} 
\usepackage{tikz}
\usetikzlibrary{matrix,arrows,shapes}

\numberwithin{equation}{section}
\newtheorem{theorem}[equation]{Theorem}

\newtheorem{lemma}[equation]{Lemma}
\newtheorem{corollary}[equation]{Corollary}

\newtheorem{notation}[equation]{Notation}

\theoremstyle{definition}
\newtheorem{rmk}[equation]{Remark}
\newenvironment{remark}[1][]{\begin{rmk}[#1] \pushQED{\qed}}{\popQED \end{rmk}}
\newtheorem{eg}[equation]{Example}
\newenvironment{example}[1][]{\begin{eg}[#1] \pushQED{\qed}}{\popQED \end{eg}}
\newtheorem{defn}[equation]{Definition}
\newenvironment{definition}[1][]{\begin{defn}[#1]\pushQED{\qed}}{\popQED \end{defn}}
\newtheorem{ques}[equation]{Question}

\usepackage[colorlinks=true, pdfstartview=FitV, linkcolor=black, citecolor=black, urlcolor=black]{hyperref}

\newcommand{\ZZ}{\mathbb{Z}}
\newcommand{\CC}{\mathbb{C}}

\renewcommand{\phi}{\varphi}
\renewcommand{\emptyset}{\varnothing}

\renewcommand{\tilde}[1]{\widetilde{#1}}

\newcommand{\setst}[2]{\left\{ #1 \mid #2 \right\}}

\newcommand{\comment}[1]{}

\makeatletter
\def\Ddots{\mathinner{\mkern1mu\raise\p@
\vbox{\kern7\p@\hbox{.}}\mkern2mu
\raise4\p@\hbox{.}\mkern2mu\raise7\p@\hbox{.}\mkern1mu}}
\makeatother

\newcommand{\Hom}{\operatorname{Hom}}

\DeclareMathOperator{\supp}{Supp}

\DeclareMathOperator{\rep}{rep}

\newcommand{\onto}{\twoheadrightarrow}

\usepackage{tipa}
\usepackage[normalem]{ulem}

\newcommand{\RR}{\mathbb{R}}
\newcommand{\ddim}{\underline{\dim}\,}
\newcommand{\kk}{\Bbbk}

\newcommand{\ses}[3]{0 \to #1 \to #2 \to #3 \to 0}

\begin{document}
\title[Total stability for type $\mathbb{A}$ quivers]{Total stability functions for type $\mathbb{A}$ quivers}
\author{Ryan Kinser}
\address{University of Iowa, Department of Mathematics, Iowa City, USA}
\email[Ryan Kinser]{ryan-kinser@uiowa.edu}

\begin{abstract}
For a quiver $Q$ of Dynkin type $\mathbb{A}_n$, we give a set of $n-1$ inequalities which are necessary and sufficient for a linear stability condition (a.k.a. central charge)
$Z\colon K_0(Q) \to \mathbb{C}$ to make all indecomposable representations stable.
We furthermore show that these are a minimal set of inequalities defining the space $\mathcal{TS}(Q)$ of total stability conditions, considered as an open subset of $\mathbb{R}^{Q_0} \times (\mathbb{R}_{>0})^{Q_0}$.
We then use these inequalities to show that each fiber of the projection of $\mathcal{TS}(Q)$ to $(\mathbb{R}_{>0})^{Q_0}$ is linearly equivalent to $\RR \times \RR_{>0}^{Q_1}$.
\end{abstract}

\makeatletter
\@namedef{subjclassname@2020}{%
  \textup{2020} Mathematics Subject Classification}
\makeatother

\subjclass[2020]{
16G20, 05E10}

\keywords{
quiver representation, stability condition, central charge, total stability, type A}

\maketitle


\section{Introduction}
\subsection{Problem statement}
Let $Q$ be an acyclic quiver and fix an algebraically closed field $\kk$ over which all representations are taken.
A detailed recollection of terminology used in the introduction is found in Section \ref{sec:background}.

This paper concerns stability conditions for quivers as in \cite{Schofield91, King94, Rudakov97, Bridgeland07}.  We restrict our attention to \emph{linear stability conditions}, 
meaning those given by a group homomorphism $Z \colon K_0(Q) \to \CC$ such that $Z(M)$ has positive imaginary part when $M$ is a nonzero representation. 
Such a function is also known as a \emph{central charge}. 
Identifying $K_0(Q) \simeq \ZZ^{Q_0}$ by taking classes of simple representations as a basis, a stability condition can be written as
\begin{equation}\label{eq:Zform}
Z(x) = w\cdot x + (r\cdot x) i \qquad\text{for some } w \in \RR^{Q_0},\ r \in (\RR_{>0})^{Q_0}.
\end{equation}
Thus we can identify the space of stability conditions with $\RR^{Q_0} \times (\RR_{>0})^{Q_0}$. 
Such $Z$ determines the \emph{slope function}
\begin{equation}\label{eq:mudef}
\mu_Z \colon K_0(Q) \to \RR, \qquad \mu_Z(x) = \frac{w\cdot x}{r\cdot x}.
\end{equation}
(Note that this is the reciprocal of the slope of the line through the origin and $\mu_Z(x)$ if we take the standard convention of drawing the complex plane.)

A representation of $Q$ is called $Z$-\emph{stable} if $\mu_Z(W) < \mu_Z(V)$ for all nonzero, proper subrepresentations $0< W<V$.
Notice that by taking $Z$ as a variable, each $\mu_Z(W) < \mu_Z(V)$ is a quadratic inequality on the space of stability conditions
$\RR^{Q_0} \times (\RR_{>0})^{Q_0}$.
Stability of quiver representations has connections with many other notions in mathematics and mathematical physics, such as moduli spaces of representations, semi-invariants, Harder-Narasimhan filtrations, and green sequences and paths.
We refer the reader to \cite{Igusa-linearity} and the references therein for more detail about these connections.

\medskip

This work investigates the set of stability conditions $Z$ such that every indecomposable representation of $Q$ is $Z$-stable.  This immediately restricts our attention to $Q$ of Dynkin type since stable representations have 1-dimensional endomorphism ring.

\begin{definition}\label{def:weights}
A stability condition $Z$ for a quiver $Q$ is \emph{totally stable}, or a \emph{total stability condition}, if every indecomposable representation of $Q$ is $Z$-stable.  The \emph{set of total stability conditions} for $Q$ is denoted
\begin{equation}
\mathcal{TS}(Q)=\setst{Z \in \Hom_\ZZ(K_0(Q), \CC)}{Z \text{ is a total stability condition}},
\end{equation}
and identified with an open subset of $\RR^{Q_0} \times (\RR_{>0})^{Q_0}$.
\end{definition}

The main results of this paper are about quivers of \emph{Dynkin type $\mathbb{A}$}.  
This means that the underlying undirected graph is of the form
\begin{equation}\label{eq:typeA}
\vcenter{\hbox{\begin{tikzpicture}[point/.style={draw,shape=ellipse,scale=.9pt},>=latex]
   \node[point] (1) at (0,0) {$1$};
  \node[point] (2) at (1,0) {$2$};
   \node[point] (3) at (2,0) {$3$};
   \node[outer sep=-2pt] (3a) at (3,0) {};
   \node[outer sep=-2pt] (d) at (3.5,0) {${\cdots}$};
   \node[outer sep=-2pt] (4a) at (4,0) {};
  \node[point] (n) at (5,0) {$n$};
  \path[-]
  	(2) edge node[auto] {} (1) 
	(3) edge node[below] {} (2) 
	(3a) edge node[below] {} (3) 
	(4a) edge node[auto] {} (n); 
	   \end{tikzpicture}}},
\end{equation}
and we say the quiver is of type $\mathbb{A}_n$ if we want to specify that it has $n$ vertices.
For \emph{equioriented} type $\mathbb{A}$ quivers, meaning all arrows point in the same direction, it is easy to describe the set of total stability conditions
due to the fact that the all indecomposable representations are uniserial in this case \cite[Example~A]{ReinekeHNsystem}.

For $Q$ of type $\mathbb{A}$ and arbitrary orientation,
$\mathcal{TS}(Q)$ was recently shown to be nonempty in independent papers of Apruzzese-Igusa \cite{AI} and Huang-Hu \cite{HH}, using quite different methods.  These papers considered the specific case of \emph{standard linear stability conditions} as in \cite{ReinekeHNsystem}, also called \emph{classical slope functions}, where $r=(1,1, \dotsc, 1)$ in \eqref{eq:Zform}.
In \cite{AI} it is proven using a geometric model, and that paper also contains more general results about affine type $\mathbb{A}$ and maximal green sequences.
Another proof was given by a different geometric model in \cite[Thm.~5.3]{BGMS}.
The methods of this paper are independent of the above cited papers,
and describe all total (linear) stability conditions, not just the standard ones.
It would be interesting to interpret our Theorem \ref{thm:main} in the geometrical models referenced above.

\subsection{Results}
The following notation for type $\mathbb{A}$ quivers is useful to organize the proof of the main theorem.  A running example illustrating the notation starts with Example \ref{ex:1}.

\begin{notation}\label{not:embed}
Given a type $\mathbb{A}$ quiver $Q$ as in \eqref{eq:typeA},  recursively define functions $x, y \colon Q_0 \to \RR$ by setting $x(1) = y(1) = 0$, and then for $i > 1$:
\begin{equation}
\begin{cases}
x(i+1) = x(i)+1\ \text{and}\ y(i+1) = y(i)  & \text{if there is an arrow } i \to i+1,\\
x(i+1) = x(i)\ \text{and}\ y(i+1) = y(i)+1 & \text{if there is an arrow }i+1 \to i.
\end{cases}
\end{equation}
(Visually, these give us an embedding  $Q\subset \RR^2$ by specifying the $x, y$-coordinates of the vertices and then connecting them with arrows in the simplest way; see \eqref{eq:ex1Q}).

This determines two sequences of subsets of $Q_0$, which are pairwise disjoint within each sequence:
\begin{equation}
X_k = \setst{z \in Q_0}{ x(z)=k},\quad Y_k = \setst{z \in Q_0}{ y(z)=k}, \quad \text{for }k \in \ZZ_{\geq 1}.
\end{equation}
We furthermore define 
\begin{equation}
\tilde{X}_i := \bigcup_{k=i}^{x(n)} X_{k} \qquad \text{and}\qquad \tilde{Y}_i := \bigcup_{k=i}^{y(n)} Y_{k}
\end{equation}
to get chains of subsets of $Q_0$:
\begin{equation}\label{eq:xychains}
\begin{split}
\tilde{X}_{x(n)} \subset \tilde{X}_{x(n)-1} \subset \dots \subset \tilde{X}_{2} \subset \tilde{X}_{1} \subset \tilde{X}_{0} = Q_0\\
Q_0 = \tilde{Y}_{0} \supset \tilde{Y}_{1} \supset \tilde{Y}_{2} \supset \dots \supset \tilde{Y}_{y(n)-1} \supset \tilde{Y}_{y(n)}.
\end{split}
\end{equation}
\end{notation}

\begin{example}\label{ex:1}
The quiver below shows a type $\mathbb{A}$ quiver embedded in $\RR^2$ as in Notation \ref{not:embed}.
\begin{equation}\label{eq:ex1Q}
Q=\quad\vcenter{\hbox{\begin{tikzpicture}[yscale=0.6,point/.style={draw,shape=ellipse,scale=.9pt},>=latex]
   \node[point] (z1) at (0,0) {$1$};
  \node[point] (z2) at (1,0) {$2$};
   \node[point] (z3) at (2,0) {$3$};
   \node[point] (z4) at (2,2) {$4$};
   \node[point] (z5) at (3,2) {$5$};
   \node[point] (z6) at (4,2) {$6$};
   \node[point] (z7) at (4,4) {$7$};
   \node[point] (z8) at (5,4) {$8$};
  \path[->]
	(z1) edge node[below,pos=0.3] {} (z2)
  	(z2) edge node[below,pos=0.3] {} (z3) 
  	(z4) edge node[below,pos=0.3] {} (z3) 
	(z4) edge node[below,pos=0.3] {} (z5)
	(z5) edge node[below,pos=0.3] {} (z6)
	(z7) edge node[below,pos=0.3] {} (z6)
	(z7) edge node[below,pos=0.3] {} (z8);
   \end{tikzpicture}}}
\end{equation}
The corresponding partitions of $Q_0$ come from vertically and horizontally aligned subsets of $Q_0$:
\begin{equation*}
\begin{split}
X_0=\{1\},\ X_1=\{2\},\ X_2=\{3,4\},\ &X_3=\{5\},\ X_4=\{6,7\},\ X_5=\{8\} \\
Y_0=\{1,2,3\},\ Y_1=\{4,5,6\},\ &Y_2=\{7, 8\}.
\end{split}
\end{equation*}
The chains in \eqref{eq:xychains} come from filtering the vertices by $x$-coordinate and $y$-coordinate respectively:
\begin{equation}
\begin{split}
\{8\} \subset \{6,7,8\} \subset\{5,6,7,8\}& \subset \{3,4,\dotsc,8\} \subset \{2,\dotsc, 8\} \subset Q_0\\
Q_0 \supset \{4,5,6,7,8\} &\supset \{7,8\}. \qedhere
\end{split}
\end{equation}
\end{example}

The main result of the paper is below, characterizing membership in $\mathcal{TS}(Q)$ by a much smaller set of inequalities than the set resulting a priori from the definition.
Here, the notation $[S]$ for $S\subseteq Q_0$ means the indecomposable representation of $Q$ supported on $S$ (see Notation \ref{not:indecomp}).

\begin{theorem}\label{thm:main}
Let $Q$ be a quiver of Dynkin type $\mathbb{A}_n$ and recall Notation \ref{not:embed}.
A stability function $Z$ is in $\mathcal{TS}(Q)$ if and only if
 the $n-1$ inequalities below hold good:
\begin{equation}\label{eq:inequality1}
\mu_Z([X_0]) > \mu_Z([X_1]) > \cdots > \mu_Z([X_{x(n)}]),
\end{equation} 
\begin{equation}\label{eq:inequality2}
\mu_Z([Y_0]) < \mu_Z([Y_1]) < \cdots < \mu_Z([Y_{y(n)}]).
\end{equation} 
Furthermore, the inequalities above are a minimal set of inequalities defining $\mathcal{TS}(Q)$ as an open set of $\RR^{Q_0} \times (\RR_{>0})^{Q_0}$.
\end{theorem}

The proof of this theorem is in Section \ref{sec:proof}.
One notices, however, that Theorem \ref{thm:main} does not tell us about the solution space to this set of inequalities.
With a little more work, we obtain the following corollary,
whose proof is in Section \ref{sec:proof} as well.

\begin{corollary}\label{cor:nonempty}
Let $Q$ be a quiver of Dynkin type $\mathbb{A}$.
Viewing $\mathcal{TS}(Q) \subset \RR^{Q_0} \times \RR_{>0}^{Q_0}$ via \eqref{eq:Zform}, the projection
$\mathcal{TS}(Q) \to \RR_{>0}^{Q_0}$ sending $(w,r) \mapsto r$
is surjective with each fiber linearly equivalent to $\RR \times \RR_{>0}^{Q_1}$.
In particular, for any $r \in \RR_{>0}^{Q_0}$ there is a total stability function for $Q$ of the form \eqref{eq:Zform}.
\end{corollary}

A final remark: perhaps unsurprisingly, we expect Theorem \ref{thm:main} and Corollary \ref{cor:nonempty} to generalize only partially to other Dynkin types.

\begin{remark}\label{rem:counterexamples}
Work in progress with Yariana Diaz and Cody Gilbert generalizes Theorem \ref{thm:main} to arbitrary Dynkin type using Auslander-Reiten sequences, but without the minimality statement.
We have also found that the cone of \emph{standard} linear stability conditions, that is, the fiber over $r=(1,1,\dotsc, 1)$ in the language of Corollary \ref{cor:nonempty}, is empty
for certain orientations in Dynkin types $\mathbb{D}_n$ for all $n \geq 9$ and types $\mathbb{E}_7, \mathbb{E}_8$
(cf. \cite[Conjecture~7.1]{ReinekeHNsystem}).
\end{remark}

\subsection*{Acknowledgements}
The author thanks \O yvind Solberg for discussions about the software QPA \cite{QPA}, which was very helpful for completing this paper.  The author also thanks Yariana Diaz and Cody Gilbert for discussions on stability of representations of Dynkin quivers and for working together on the QPA and SageMath code which helped finish this work.
Special thanks go to Hugh Thomas for the proof of Corollary \ref{cor:nonempty}, and an anonymous commentor for pointing out that the results in the first version of this article used outdated language.
This work was supported by a grant from the Simons Foundation (636534, RK).


\section{Background}\label{sec:background}
In this section we establish our notation and make some initial reductions for the proof of the main theorem.
More detailed background can be found in textbooks such as \cite{Schiffler:2014aa,DWbook} and the survey \cite{Reinekesurvey}.

\subsection{Quiver representations}
We write $Q_0$ for the set of vertices of a quiver $Q$, and $Q_1$ for its set of arrows, while $t\alpha$ and $h\alpha$ denote the tail and head of an arrow $t\alpha \xrightarrow{\alpha} h\alpha$. A \emph{representation} $V$ of $Q$ assigns a finite-dimensional vector space $V(z)$ to each $z\in Q_0$, and to each $\alpha \in Q_1$ a choice of linear map $V(\alpha) \colon V(t\alpha) \to V(h\alpha)$.  
A \emph{subrepresentation} $W \subseteq V$ is a collection of subspaces $(W(z) \subseteq V(z))_{z \in Q_0}$ such that $V(\alpha)(W(t\alpha)) \subseteq W(h\alpha)$ for all $\alpha \in Q_1$.
The \emph{support} of a representation $V$, written $\supp V$, is the set of vertices $z \in Q_0$ such that $V(z) \neq 0$.
Definitions of standard notions such as \emph{morphisms}, \emph{direct sum}, and \emph{indecomposability} can be found in the references above.

\begin{notation}\label{not:indecomp}
For a subset $S \subseteq Q_0$, we let $[S]$ be the representation $Q$ such that
\begin{equation}
[S](z) = 
\begin{cases}
\kk & z \in S\\
0 & z \not \in S
\end{cases}
\quad \text{and} \quad
[S](\alpha) = 
\begin{cases}
id_\kk & t\alpha,\, s\alpha \in S\\
0 & \text{otherwise}.
\end{cases}
\end{equation}
\end{notation}

For a quiver of type $\mathbb{A}_n$, it can be seen from repeated use of Gaussian elimination that 
as $S$ varies over all intervals $S=\{i, \dotsc, j\}$ for $1 \leq i \leq j \leq n$,
the representations $[S]$ trace out all isomorphism classes of indecomposables for type $\mathbb{A}$ quivers (a special case of Gabriel's theorem \cite{gabriel}).

Assume $Q$ is acyclic.  We write $\rep(Q)$ for the category of representations of $Q$, and often identify the \emph{Grothendieck group} $K_0(Q):=K_0(\rep(Q))$ with $\ZZ^{Q_0}$ by identifying the class of the simple representation $[\{z\}]$, for $z \in Q_0$, with the standard basis vector which is 1 in coordinate $z$ and 0 elsewhere.


\subsection{Stability}
A \emph{weight} on a quiver $Q$ is an element $w \in \RR^{Q_0}$, where we write $w_z\in \RR$ for the value in coordinate $z \in Q_0$, and we write $|w| :=\sum_{z \in Q_0}w_z$.
A \emph{dimension vector} for $Q$ is a weight such that each $w_z$ is a nonnegative integer.  The dimension vector of a representation $V$ of $Q$ is  $(\dim V(z))_{z \in Q_0}$.
Given a weight $w$ and representation $V$ of $Q$, we write $w(V):=w \cdot \ddim V$, where $\cdot$ is the standard dot product on $\RR^{Q_0}$.

\begin{remark}
Another notion of stability which is prevalent in quiver literature is the following \cite{King94}.
A representation $V$ of $Q$ is \emph{$\theta$-stable} for $\theta \in \RR^{Q_0}$ if $\theta(V)=0$ and $\theta(W) <0$ for all proper, nonzero subrepresentations $W < V$.  
It can be directly seen that a representation which is $\theta$-stable is $Z$-stable as well for any $Z$ of the form \eqref{eq:Zform} with $w=\theta$, but the converse does not hold.  However, we can take an arbitrary $Z$ as in \eqref{eq:Zform} and define the weight
\begin{equation}
\theta:=r(V)\,w - w(V)r,
\end{equation}
and we have that $V$ is $Z$-stable if and only if $V$ is $\theta$-stable.

The $\theta$-stable representations for fixed $\theta$ are the simple objects of the full, abelian subcategory of $\theta$-semistable representations inside the category of all finite-dimensional representations of $Q$.  Thus we can never have all indecomposable representations $\theta$-stable in the above sense if $Q$ has a nonempty arrow set (by Schur's lemma).
 \end{remark}

\subsection{Initial reductions}
The results of this section are valid for all quivers, not just type $\mathbb{A}$. Presumably these lemmas have been observed elsewhere, but we include proofs of everything for completeness.
The following lemma can be easily checked (e.g. the proof of \cite[Lemma~4.1]{Reinekesurvey} generalizes).

\begin{lemma}\label{lem:seesaw}
Let $Z$ be a linear stability function on $Q$ and $0 \to A \to B \to C \to 0$ a short exact sequence in $\rep(Q)$.  Then we have the ``seesaw property'':
\begin{equation}
\mu_Z(A) < \mu_Z(B)\quad \Leftrightarrow \quad \mu_Z(B) < \mu_Z(C)\quad \Leftrightarrow \quad \mu_Z(A) < \mu_Z(C),
\end{equation}
and the same is true when $<$ is replaced by $\leq$.  

Furthermore, we have for any $c \in \RR$ that
\begin{equation}\label{eq:Zmin}
\mu_Z(A),\ \mu_Z(C) < c \quad \Rightarrow \quad \mu_Z(B) < c.
\end{equation}
\end{lemma}

The next lemma allows us to only consider \emph{indecomposable} subrepresentations to determine if a representation is stable.

\begin{lemma}\label{lem:indecomp}
Let $Q$ be an arbitrary quiver.  A representation $V$ of $Q$ is $Z$-stable if and only if $\mu_Z(W) < \mu_Z(V)$ for all proper nonzero indecomposable subrepresentations $W < V$.
\end{lemma}
\begin{proof}
The forward implication follows from the definition.  For the converse, assume $\mu_Z(W) < \mu_Z(V)$ for all proper indecomposable subrepresentations $W < V$, and let $Y < V$ be an arbitrary proper nonzero subrepresentation.  
If every indecomposable summand of $Y$ were to have slope strictly less than $Y$, this would contradict Lemma \ref{lem:seesaw}.
Therefore, taking $W \leq Y$ to be an indecomposable summand of $Y$ of maximal slope, we have  $\mu_Z(Y) \leq \mu_Z(W)$, and the result follows.
\end{proof}

Recall that from a quiver $Q$ we obtain its \emph{opposite quiver} $Q^{\rm op}$ by reversing the orientations of all arrows of $Q$.  Note that a central charge for $Q$ is also one for $Q^{\rm op}$ via the isomorphism $K_0(Q) \simeq K_0(Q^{\rm op})$ obtained by identifying the simple representations of $Q$ and $Q^{\rm op}$ associated to the same vertex.
Taking the vector space dual at each vertex, and dual map over each arrow, gives a duality between the categories of representations of $Q$ and $Q^{\rm op}$.
The following lemma gives a helpful connection between stability in these categories.

\begin{lemma}\label{lem:dual}
Let $Q$ be an arbitrary quiver and $Z$ a central charge for $Q$ and $Q^{\rm op}$.  Let $W < V$ be a proper nonzero subrepresentation and $(V/W)^* < V^*$ the corresponding dual subrepresentation of $Q^{\rm op}$.  Then $\mu_Z(W) < \mu_Z(V)$ if and only if $\mu_Z( (V/W)^*) < \mu_Z( V^*)$.
\end{lemma}
\begin{proof}
Noting that $\mu_Z(X)=\mu_Z(X^*)$ for any representation $X$ of $Q$, this lemma follows from Lemma \ref{lem:seesaw}.
\end{proof}


The following lemma is used to prove the minimality part of the main theorem.
For $r \in \RR_{>0}^{Q_0}$, we denote by $\mathcal{TS}_r(Q)$ the fiber of the projection described in Corollary \ref{cor:nonempty}:
\begin{equation}\label{eq:TSrQ}
\mathcal{TS}_r(Q)=\setst{w \in \RR^{Q_0}}{Z(x)= w\cdot x + (r\cdot x) i \in \mathcal{TS}(Q)}.
\end{equation}

\begin{lemma}\label{lem:lineality}
Let $Q$ be a connected quiver and $r \in \RR_{>0}^{Q_0}$ such that $\mathcal{TS}_r(Q)$ is nonempty.  
Then the only subspace of $\RR^{Q_0}$ which has a translate contained in $\mathcal{TS}_r(Q)$ is $\RR r$.
\end{lemma}
\begin{proof}
For $w \in \mathcal{TS}_r(Q)$ and $c \in \RR$, let $\mu_{w+cr}$ be the slope function associated to the central charge $Z(x) = (w + cr)\cdot x + (r\cdot x)i$.
It can be directly checked that $\mu_{w+cr} = \mu_w + c$ as $\RR$-valued functions, and thus $w \in \mathcal{TS}_r(Q)$ 
if and only if $w+cr \in \mathcal{TS}_r(Q)$.  So if $\mathcal{TS}_r(Q)$ is nonempty, a translate of the line $\RR r$ is contained in $\mathcal{TS}_r(Q)$.

Now suppose for contradiction that there exists $w \in \mathcal{TS}_r(Q)$ and $\eta \in \RR^{Q_0}$, where $\RR \eta \neq \RR r$, 
but that the affine linear subspace $w+\RR \eta \subset \mathcal{TS}_r(Q)$.
Since $Q$ is connected, $\RR \eta \neq \RR r$ implies that there exists $\alpha \in Q_1$ such that $\eta_{t\alpha}r_{t\alpha} \neq \eta_{h\alpha}r_{h\alpha}$.
Let $V$ be the indecomposable representation which is dimension 1 on $t\alpha, h\alpha$ and 0 elsewhere, whose only nonzero map is over $\alpha$, and let $W<V$ the simple subrepresentation supported at $\{h\alpha\}$.  Then we compute
\begin{equation}
\mu_{\eta}(V)-\mu_{\eta}(W) = 
\frac{\eta_{t\alpha}+\eta_{h\alpha}}{r_{t\alpha} + r_{h\alpha}} - \frac{\eta_{h\alpha}}{r_{h\alpha}} = 
\frac{r_{h\alpha}\eta_{t\alpha}-r_{t\alpha}\eta_{h\alpha}}{(r_{h\alpha} + r_{h\alpha})r_{h\alpha}} \neq 0 .
\end{equation}
Since we assumed $w+\RR \eta \subset \mathcal{TS}(Q)$, for any $c \in \RR$ we have
\begin{equation}
0 < \mu_{w+ c\eta}(V)-\mu_{w + c\eta}(W) = (\mu_{w}(V) - \mu_{w}(W)) + c(\mu_{\eta}(V)-\mu_{\eta}(W)),
\end{equation}
a contradiction since $c$ is arbitrary and the other values on the right hand side are fixed.
\end{proof}

\section{Proof of the main theorem}\label{sec:proof}
We begin with an elementary lemma that is used repeatedly throughout the proof of the main theorem.

\begin{lemma}\label{lem:chains}
A stability function $Z$ satisfies the sequence of inequalities \eqref{eq:inequality1} if and only if it satisfies
\begin{equation}\label{eq:Avgchain1}
\mu_Z([\tilde{X}_0]) > \mu_Z([\tilde{X}_1]) > \cdots > \mu_Z([\tilde{X}_{x(n)}]),
\end{equation} 
Similarly, $Z$ satisfies the sequence of inequalities \eqref{eq:inequality2} if and only if it satisfies
\begin{equation}\label{eq:Avgchain2}
\mu_Z([\tilde{Y}_0]) < \mu_Z([\tilde{Y}_1]) < \cdots < \mu_Z([\tilde{Y}_{y(n)}]).
\end{equation} 
\end{lemma}
\begin{proof}
The two filtrations of $Q_0$ in \eqref{eq:xychains} induce sequences of morphisms in $\rep(Q)$:
\begin{equation}\label{eq:xyrepchains}
\begin{split}
[\tilde{X}_{x(n)}] \subset [\tilde{X}_{x(n)-1}] \subset \dots \subset [\tilde{X}_{2}] \subset [\tilde{X}_{1}] \subset [\tilde{X}_{0}] = [Q_0]\\
[Q_0] = [\tilde{Y}_{0}] \onto [\tilde{Y}_{1}] \onto [\tilde{Y}_{2}] \onto \dots \onto [\tilde{Y}_{y(n)-1}] \onto [\tilde{Y}_{y(n)}].
\end{split}
\end{equation}
Applying Lemma \ref{lem:seesaw} finishes the proof.
\end{proof}

The following elementary observation is useful in carrying out proof by induction for the main theorem.

\begin{lemma}\label{lem:topsupports}
Let $Q$ be a type $\mathbb{A}_n$ quiver, and $V$ an indecomposable representation of $Q$ with $n \in \supp V$.  Then there exists $k$ such that either $V = [\tilde{X}_k]$ or $V = [\tilde{Y}_k]$.
\end{lemma}

\subsubsection*{Proof of ``if and only if'' statement of Theorem \ref{thm:main}}
The $\Rightarrow$ direction of follows immediately from Lemma \ref{lem:chains} and the sequences of morphisms in \eqref{eq:xyrepchains}.

For the $\Leftarrow$ direction, we assume that $Z$ is given such that the inequalities \eqref{eq:inequality1} and  \eqref{eq:inequality2} all hold, and we need to show that $Z \in \mathcal{TS}(Q)$.
By Lemma \ref{lem:indecomp}, we are reduced to showing $\mu_Z(W) < \mu_Z(V)$ for all $0 <W<V$ with both $W, V$ indecomposable.
We observe for later that since $\tilde{X}_{0} = Q_0 = \tilde{Y}_{0}$, we may concatenate the chains \eqref{eq:Avgchain1} and \eqref{eq:Avgchain2} to obtain:
\begin{equation}\label{eq:longchain}
\mu_Z([\tilde{X}_{x(n)}])< \cdots < \mu_Z([\tilde{X}_{1}]) < \mu_Z([Q_0]) < 
\mu_Z([\tilde{Y}_{1}]) < \cdots < \mu_Z([\tilde{Y}_{y(n)}]).
\end{equation}

We use induction on the number of vertices of $Q$, with the statement being vacuously true in the base case $n=1$ (the unique indecomposable is stable with respect to any $Z$, and there are no inequalities to satisfy). Let $Q$ be a type $\mathbb{A}_n$ quiver and assume the theorem is true for type $\mathbb{A}_{n-1}$ quivers. 
The primary challenge in the induction is that the collection of inequalities \eqref{eq:inequality1}, \eqref{eq:inequality2} for $Q$ does not simply restrict to the corresponding collection of inequalities for smaller quivers, so we cannot easily apply the induction hypothesis.

Consider the arrow $n-1 \to n$ in $Q$: we can assume $n$ is a sink without loss of generality because Lemma \ref{lem:dual} gives us the $n-1 \leftarrow n$ case from this by reversing the directions of all inequalities, noting that the sets $X_k$ and $Y_k$ are interchanged when switching between $Q$ and $Q^{\rm op}$.
Thus we have
\[
x(n) = x(n-1)+1 \quad \text{and} \quad y(n) = y(n-1).
\]
Let $\overline{Q}$ be the quiver obtained by removing vertex $n$ and the arrow connected to it, and $\bar Z \in \Hom_\ZZ(K_0(\overline{Q}), \CC)$ be the restriction of $Z$ to $K_0(\overline{Q})$. 
We furthermore institute superscripts to distinguish between objects associated to $Q$ and $\overline{Q}$, whenever necessary.

To apply the induction hypothesis to $\overline{Q}$, we need to show that $\bar Z$ satisfies the sequences of inequalities in \eqref{eq:inequality1} and \eqref{eq:inequality2} associated to $\overline{Q}$, namely:
\begin{equation}\label{eq:barinequality1}
\mu_{\bar Z}([X^{\overline{Q}}_0]) > \mu_{\bar Z}([X^{\overline{Q}}_1]) > \cdots > \mu_{\bar Z}([X^{\overline{Q}}_{x(n-1)}]),
\end{equation}
\begin{equation}\label{eq:barinequality2}
\mu_{\bar Z}([Y^{\overline{Q}}_0]) < \mu_{\bar Z}([Y^{\overline{Q}}_1]) < \cdots < \mu_{\bar Z}([Y^{\overline{Q}}_{y(n-1)}]).
\end{equation}
Whenever $n \notin S\subseteq Q_0$, the function $\mu_Z([S])$ is independent of the $n^{\rm th}$ coordinate, and can thus be identified with the function $\mu_{\bar Z}([S])$ of $\bar Z$ on the space of stability conditions for $\overline{Q}$.  Since $X^{\overline{Q}}_k =X^{Q}_k$ for $0 \leq k \leq x(n-1)$ and $Y^{\overline{Q}}_k =Y^{Q}_k$ for $0 \leq k \leq y(n-1)-1$, we know $\bar Z$ satisfies all the inequalities in \eqref{eq:barinequality1} and \eqref{eq:barinequality2}, except perhaps the far right inequality of \eqref{eq:barinequality2}, where we must deal with the fact that $Y^{\overline{Q}}_{y(n-1)} = Y^Q_{y(n-1)} \setminus\{n\}$.

Thus to apply the induction hypothesis, it remains to show that the far right inequality of \eqref{eq:barinequality2} holds, which can be written as
\begin{equation}\label{eq:IHneeded}
\mu_{\bar Z}([Y^{\overline{Q}}_{y(n)-1}])  < \mu_{\bar Z}([Y^{\overline{Q}}_{y(n)}]),
\end{equation}
where we use $y(n-1)= y(n)$ to simplify the notation here and below.
Recalling that $\tilde{X}^Q_{x(n)} = \{n\}$ since $n$ is a sink, from \eqref{eq:longchain} we can extract
\begin{equation}\label{eq:IH1}
\mu_Z([\{n\}])=\mu_Z([\tilde{X}^Q_{x(n)}]) < \mu_Z([\tilde{Y}^Q_{y(n)-1}]).
\end{equation}
From the short exact sequence
$\ses{[\{n\}]}{[\tilde{Y}^Q_{y(n)-1}]}{[\tilde{Y}^{\overline{Q}}_{y(n)-1}]}$ and \eqref{eq:IH2} we obtain
\begin{equation}\label{eq:IH2}
\mu_Z([\tilde{Y}^Q_{y(n)-1}]) < \mu_Z([\tilde{Y}^{\overline{Q}}_{y(n)-1}]).
\end{equation}
From the short exact sequence
$\ses{[Y^Q_{y(n)-1}]}{[\tilde{Y}^Q_{y(n)-1}]}{[Y^Q_{y(n)}]}$ 
and \eqref{eq:inequality2} we obtain
\begin{equation}\label{eq:IH3}
\mu_Z([Y^Q_{y(n)-1}]) < \mu_Z([\tilde{Y}^Q_{y(n)-1}]).
\end{equation}
Finally, from the short exact sequence
$\ses{[Y^Q_{y(n)-1}]}{[\tilde{Y}^{\overline{Q}}_{y(n)-1}]}{[Y^{\overline{Q}}_{y(n)}]}$ 
and combining \eqref{eq:IH3} then \eqref{eq:IH2} we obtain
\begin{equation}\label{eq:IH4}
\mu_Z([Y^Q_{y(n)-1}]) < \mu_Z([Y^{\overline{Q}}_{y(n)}]).
\end{equation}
Since, $Y^Q_{y(n)-1} = Y^{\overline{Q}}_{y(n)-1}$,
restricting the domain to $\bar{Z}$ gives exactly the inequality \eqref{eq:IHneeded} we set out to show in this paragraph.

Now by the induction hypothesis, $\bar Z$ satisfies all inequalities $\mu_{\bar Z}(W) <\mu_{\bar Z}(V)$ for $V$ an indecomposable representation of $\overline{Q}$ and $0 \neq W < V$.  
This means $Z$ satisfies all such inequalities when $n \notin \supp V$.
So it remains to consider indecomposable $V$ which are supported at vertex $n$.

\medskip

We first consider the case when $n \in \supp W$ as well, setting out to show $\mu_Z(W) < \mu_Z(V)$.
Let $[n]:=[\{n\}]$ and $\overline{W} := W/[n]$ and $\overline{V} :=V/[n]$.
The case where $W=[n]$ follows from the chain \eqref{eq:longchain} since its least term is $\mu_Z([n])$, and $\mu_Z(V)$ must appear in this chain by Lemma \ref{lem:topsupports}.
So we may assume now that $\overline{W} \neq 0$.
By the seesaw property, it is enough to show $\mu_Z(W) < \mu_Z(V/W)=\mu_Z(\overline{V}/\overline{W})$, and for this it is enough to show both:
\begin{equation}
\mathrm{(i)}\ \mu_Z(\overline{W}) < \mu_Z(\overline{V}/\overline{W})
\qquad \text{and}\qquad 
\mathrm{(ii)}\ \mu_Z([n]) < \mu_Z(\overline{V}/\overline{W})
\end{equation}
by the second statement of Lemma \ref{lem:seesaw} applied to the short exact sequence
$\ses{[n]}{W}{\overline{W}}$.
Inequality (i) is immediate from the induction hypothesis since $[\overline{W}]$ is a subrepresentation of $\overline{V}$.
For (ii), we have from the far right inequality of \eqref{eq:inequality1} that $\mu_Z([n]) < \mu_Z(X^Q_{x(n)-1})$.  Then noting $X^Q_{x(n)-1}=X^{\overline{Q}}_{x(n)-1}$, applying the induction hypothesis to the chain of subrepresentations below yields:
\begin{equation}
[X^{\overline{Q}}_{x(n)-1}] \leq \overline{W} < \overline{V}
\quad \Rightarrow\quad
\mu_Z([n]) < \mu_Z(\overline{W}) < \mu_Z(\overline{V}).
\end{equation}
Since $\mu_Z(\overline{V}) < \mu_Z(\overline{V}/\overline{W})$ by the seesaw property, inequality (ii) is shown, completing the case $n \in \supp W$.

\medskip

We now consider the case of pairs $W<V$ when $n \notin \supp W$ but $n\in \supp V$.  
Fix such a $V$, noting $V$ has the representation $[Y^Q_{y(n)}]$ as a proper, nonzero quotient in order for there to exist nonzero $W < V$ without $n$ in its support (i.e. otherwise $V$ would be uniserial with socle $[n]$).  
Recalling $Y^Q_{y(n)}=\tilde{Y}^Q_{y(n)}$, the chain of inequalities \eqref{eq:longchain} and Lemma \ref{lem:topsupports} imply that $\mu_Z(V) < \mu_Z([Y^Q_{y(n)}])$.
Now we have a short exact sequence
\begin{equation}
\ses{W_0}{V}{Y^Q_{y(n)}}
\end{equation}
where $W_0=[\supp V \setminus Y^Q_{y(n)}]$ is the unique maximal subrepresentation of V not supported at $n$, so the inequality just shown gives $\mu_Z(W_0) < \mu_Z(V)$ by the seesaw property.
The induction hypothesis then implies that $\mu_{Z}(W)$ is maximized at $W_0$, as $W$ runs over all subrepresentations of $V$ not supported at $n$, so we have shown that $\mu_Z(W) < \mu_Z(V)$ and the proof of the ``if and only if'' part of the main theorem is completed. \qed

\subsubsection*{Proof of minimality}
To prove minimality, we need to use that each fiber $\mathcal{TS}_{r}(Q)$ is nonempty. This is proven independently to this proof in Corollary \ref{cor:nonempty} below, so let us assume it for now.  For a fixed $r \in \RR_{>0}^{Q_0}$, we have just proven that the $n-1$ inequalities \eqref{eq:inequality1} and \eqref{eq:inequality2} cut out the cone $\mathcal{TS}_{r}(Q)\subset \RR^{Q_0}$.  If any of them could be omitted, then $\mathcal{TS}_r(Q)$ could be represented as the intersection of $n-2$ or fewer linear half spaces.
But then $\mathcal{TS}_r(Q)$ would contain a translate of a two-dimensional subspace of $\RR^{n}$, contradicting Lemma \ref{lem:lineality}. \qed

\subsubsection*{Proof of Corollary \ref{cor:nonempty}} 
 This proof is due to Hugh Thomas.  We begin by setting
\begin{equation}
x_i :=\sum_{k \in X_i^Q} r_k , \qquad y_i :=\sum_{k \in Y_i^Q} r_k, \qquad \tilde{x}_i := \sum_{k=1}^i x_k, \qquad \tilde{y}_i := \sum_{k=1}^i y_k.
\end{equation}
For $Z$ as in \eqref{eq:Zform} for $r$ fixed, we consider the linear functions of $w$ defined by 
\begin{equation}
\begin{split}
f_i(w) &:= \mu_{Z}([X_i^Q]) - \mu_{Z}([X_{i+1}^Q]), \qquad 1 \leq i \leq M:=\max\{i : X_{i+1}^Q \neq \emptyset \}\\
g_j(w) &:= \mu_{Z}([Y_{j+1}^Q]) - \mu_{Z}([Y_{j}^Q]), \qquad 1 \leq j \leq N:=\max\{j : Y_{j+1}^Q \neq \emptyset \}.
\end{split}
\end{equation}
Our main theorem says that $Z\in \mathcal{TS}_r(Q)$ if and only if these functions are all strictly positive on the weight.  
We can assume that both $M, N \geq 1$, since otherwise the quiver is equioriented and the corollary is immediate \cite[Example~A]{ReinekeHNsystem}.

If $\mathcal{TS}_r(Q)$ were empty, then by Farkas' lemma (see for example \cite[\S5.8.3]{BV04}) there would exist a linear combination with nonnegative coefficients 
\begin{equation}\label{eq:zerolincomb}
0= \sum_{i=1}^M a_i f_i(w) + \sum_{j=1}^N b_j g_j(w), \quad a_i, b_j \in \mathbb{R}_{\geq 0},
\end{equation}
where some $a_i \neq 0$ for $1 \leq i \leq M$ or some $b_j \neq 0$ for $1 \leq j \leq N$.
Assume for contradiction that we have such an expression, and take one for which $Q$ has a minimal number of vertices.  We will successively consider the coefficients of $w_1, w_2, w_3, \dotsc$ and show that (up to a scalar multiple) the vanishing of these coefficients forces $a_i = \tilde{x}_i$ and $b_j = \tilde{y}_j$ up to a point, and then yields a contradiction when considering the coefficient of $w_t$ when either $x(t)$ or $y(t)$ is maximal (i.e., in Notation \ref{not:embed}, when we reach a vertex in the furthest right column of vertices or furthest up row of vertices).

First consider the coefficient of $w_1$.  Assume $1 \to 2$ in $Q$ (without loss of generality by the same application of Lemma \ref{lem:dual} used in the proof of the main theorem).  This variable appears only in $f_1(w)$ and $g_1(w)$, and the coefficient of $w_1$ in \eqref{eq:zerolincomb} is $\frac{a_1}{x_1} - \frac{b_1}{y_1}$.  Up to a scalar, we are forced to take $a_1 = r_1 = \tilde{x}_1$ and $b_1 = y_1 = \tilde{y}_1$.  

Proceeding inductively up the indices for $w$, we next consider the coefficient of $w_t$ for $1< t< n$ but $y(t)=1$ still (i.e., we have a path $1 \to 2\to \cdots \to t$ in $Q$).
The coefficient of $w_t$ in \eqref{eq:zerolincomb} receives contributions from (at most) 
$f_{t-1}(w),\ f_t(w),\ g_1(w)$.
If $x(t)$ is not maximal, then $t \leq M$  and for \eqref{eq:zerolincomb} to hold we need
\begin{equation}\label{eq:cor1}
0 = a_{t-1}\frac{-1}{x_t} + a_t \frac{1}{x_t} + b_{1}\frac{-1}{y_{1}}.
\end{equation}
By induction we already have $a_{t-1} =\tilde{x}_{t-1}(=t-1)$ and $b_1 = \tilde{y}_1$, so a direct substitution into the above expression yields
\begin{equation}
\frac{1}{x_t}(a_t - \tilde{x}_{t-1}) -1 = 0,
\end{equation}
forcing $a_{t} = \tilde{x}_{t-1}+x_{t} = \tilde{x}_{t}$.
However, if $x(t)$ is maximal, then $M=t-1$ so for \eqref{eq:zerolincomb} to hold we need
\begin{equation}
a_{t-1}\frac{-1}{x_t} + b_{1}\frac{-1}{y_{1}} = 0,
\end{equation}
which is a contradiction since both terms of the left hand side are negative.

Continuing up the indices, consider the general situation of $t\in Q_0$ such that both $k:=x(t)>1$ and $l:=y(t)>1$ and neither is maximal among vertices of $Q$.
The coefficient of $w_t$ in \eqref{eq:zerolincomb} receives contributions from 
$f_{k-1}(w),\, f_k(w),\, g_{l-1}(w),\, g_{l}(w)$, and \eqref{eq:zerolincomb} implies
\begin{equation}\label{eq:corproof2}
a_{k-1}\frac{-1}{x_k} + a_k \frac{1}{x_k} + b_{l-1}\frac{1}{y_l} + b_l\frac{-1}{y_l} = 0.
\end{equation}
By induction on $t$ we already have $a_{k-1} =\tilde{x}_{k-1}$ and $b_{l-1} = \tilde{y}_{l-1}$, and either $a_k = \tilde{x}_k$ (if $t-1 \leftarrow t$ in $Q$)
or $b_l = \tilde{y}_l$ (if $t-1 \to t$ in $Q$), so the remaining coefficient is determined in \eqref{eq:corproof2}.
In the case that $t-1\leftarrow t$ in $Q$, direct substitution into the above expression yields
\begin{equation}
1  + \frac{1}{y_{l}}(\tilde{y}_{l-1} - b_l) = 0,
\end{equation}
forcing $b_l = \tilde{y}_{l-1}+y_l = \tilde{y}_l$.  The case that $t-1 \to t$ in $Q$ is similar.

At some point we arrive at $t\in Q_0$ such that either $k$ or $l$ is maximal, say $k$ (again the other case is similar).  Then the arrows of $Q$ are oriented like $t-1 \to t \leftarrow \cdots \leftarrow n$.  The coefficient of $w_t$ has one fewer term and is by induction equal to
\begin{equation}
a_{k-1}\frac{-1}{x_k} +  b_{l-1}\frac{1}{y_l} + b_l\frac{-1}{y_l} = \tilde{x}_{k-1}\frac{-1}{x_k} -1 < 0,
\end{equation}
thus nonvanishing.  This is the desired contradiction and the corollary is proven.
\qed

\bigskip

We illustrate the main theorem by continuing our running example.

\begin{example}\label{ex:2}
Continuing Example \ref{ex:1}, Theorem \ref{thm:main} says that the minimal set of inequalities in the variable $Z$ which define $\mathcal{TS}(Q)$ is:
\begin{equation}\label{eq:exineq1}
\mu_Z([1]) > \mu_Z([2]) > \mu_Z([3,4]) > \mu_Z([5]) > \mu_Z([6,7]) > \mu_Z([8]),
\end{equation}
\begin{equation}\label{eq:exineq2}
\mu_Z([1,2,3]) < \mu_Z([4,5,6]) < \mu_Z([7,8]).
\end{equation}
Taking coordinates $w_1, \dotsc, w_8, r_1, \dotsc r_8$ on $\RR^{Q_0} \times (\RR_{>0})^{Q_0}$, these are explicitly
\begin{equation}
\frac{w_1}{r_1} > \frac{w_2}{r_2} > \frac{w_3+w_4}{r_3+r_4} > \frac{w_5}{r_5} > \frac{w_6+w_7}{r_6+r_7} > \frac{w_8}{r_8}
\end{equation}
\begin{equation}
\frac{w_1+w_2+w_3}{r_1+r_2+r_3} < \frac{w_4+w_5+w_6}{r_4+r_5+r_6}  < \frac{w_7+w_8}{r_7+r_8},
\end{equation}
and admit a solution in $w$ for any choice of $r$.
\end{example}

\bibliographystyle{alpha}
\bibliography{typeAstability}

\end{document}